\documentclass[11pt]{amsart}

\usepackage{fullpage}
\usepackage{amsmath,amssymb,graphicx,float,epsf}
\usepackage{epic,eepic}
\usepackage{float}
\restylefloat{figure}

\newcommand{\eeq}{\end{equation}}
\newtheorem{theorem}{Theorem}[section]

\numberwithin{equation}{section}

\newtheorem{uess}[theorem]{Lemma}
\newtheorem{guess}[theorem]{Theorem}
\newtheorem{remark}[theorem]{Remark}
\newtheorem{corollary}[theorem]{Corollary}
\newtheorem{example}[theorem]{Example}
\newtheorem{definition}[theorem]{Definition}
\newtheorem{proposition}[theorem]{Proposition}

\title{Exponential stability of systems of vector delay differential equations  
with applications to second order equations}

\author{Leonid Berezansky \\
Dept. of Math, Ben-Gurion University of the Negev, Beer-Sheva 84105, Israel \\
and Elena Braverman$^1$ \\
Dept. of Math \& Stats, University of Calgary,  2500 University Dr. NW, Calgary, AB, Canada T2N1N4}
\date{\today}

\begin{document}

\footnotetext[1]{Corresponding author. E-mail {\em
maelena@math.ucalgary.ca}. Fax (403)-282-5150. Phone (403)-220-3956.
Partially supported by the NSERC Research Grant RGPIN-2020-03934.}

\begin{abstract}
Various results and techniques, such as 
Bohl-Perron theorem, a priori solution estimates,
M-matrices and the matrix measure, are applied to obtain new explicit exponential stability
conditions for the system  of vector functional differential equations 
$$
\dot{x_i}(t)=A_i(t)x_i(h_i(t)) +\sum_{j=1}^n \sum_{k=1}^{m_{ij}} B_{ij}^k(t)x_j(h_{ij}^k(t))
+ \sum_{j=1}^n\int\limits_{g_{ij}(t)}^t K_{ij}(t,s)x_j(s)ds,~i=1,\dots,n.
$$
Here $x_i$ are unknown vector functions, $A_i, B_{ij}^k, K_{ij}$ are  matrix functions, $h_i,h_{ij}^k, g_{ij}$  are
delayed arguments.
Using these results, we deduce
explicit exponential stability tests for second order vector delay differential equations.


\textbf{Keywords:} Exponential stability; differential systems with matrix coefficients and a distributed delay; second order vector delay differential equations;  Bohl-Perron theorem; matrix measure; M-matrices.

\textbf{AMS(MOS) subject classification:} 34K20, 34K06, 34K25.

\end{abstract}

\maketitle

\section{Introduction}

Exponential stability of vector, scalar and higher order delay differential equations has attracted a lot of attention, see \cite{AS,BB1,BDSZ1,Hale,Gil,Gopalsamy,KN,KM,Kuang,Ngoc1} and the bibliography therein. Considered models included neutral equations and systems with both concentrated and distributed delays.
However,  for two classes of linear functional differential equations there are only few stability results. These types are
systems of several linear delay differential equations of the first order with matrix coefficients and systems, where at least one equation is of the second or higher order.
Scientific interest to these models is not purely theoretical.
There are many interesting real-world applications of such systems. For example, 
the ordinary differential vector equation of the second order
\begin{equation}\label{1.1}
\ddot{x}(t)+A(t)\dot{x}(t)+B(t)x(t)=0
\end{equation}
arises in control theory \cite{Harris, Rugh}.
Explicit asymptotic stability conditions for vector ordinary differential equations of the second order were obtained
in \cite{Harris, Gil_a, Rugh, Sun, Zevin}. Delay differential vector equations of the second order are  natural generalizations
of \eqref{1.1} that also can be applied in control theory. However, presently this application is not practical due to lack of qualitative results for this class of equations. We can mention only the paper \cite{Tunc1}, where for the nonlinear vector equation
\begin{equation}\label{1.2}
\ddot{x}(t)+F(x(t),\dot{x}(t))\dot{x}(t)+H(x(t-\tau))=L(t)
\end{equation}
asymptotic stability of the zero solution and the
boundedness of all solutions of equation \eqref{1.2} were investigated.
One of the goals of this paper is to fill this gap, by studying exponential stability of linear and nonlinear
vector delay differential equations of the second order.
 
The paper is organized as follows. Section 2 contains relevant 
definitions, notations and auxiliary statements. In Section 3, we obtain the main result of the paper on exponential stability for linear systems of functional differential equations with matrix coefficients. We also give  several  corollaries, one of them will further be applied to second order vector equations.
Section 4 deals with exponential stability for a linear delay differential equation of the second order with matrix coefficients. 
A stability test obtained here is one of the first stability results for this class of equations.
As a corollary, we deduce new stability conditions for non-delay vector equation \eqref{1.1}. 
In Section 5, we compare stability results of the present paper with known ones.
Section 5 also includes some open problems and topics for future research.

\section{Preliminaries}

The main object of the present paper is
\begin{equation}
\label{2.1}
\dot{x_i}(t)= 
A_i(t)x_i(h_i(t)) +\sum_{j=1}^n \sum_{k=1}^{m_{ij}} B_{ij}^k(t)x_j(h_{ij}^k(t)) 
+ \displaystyle  \sum_{j=1}^n\int\limits_{g_{ij}(t)}^t K_{ij}(t,s)x_j(s)ds,~i=1,\dots,n, ~ t \geq t_0\geq 0, 
\end{equation}
where $x_i: {\mathbb R} \to {\mathbb R}^d$, $i=1,\dots,n$ are unknown vector-functions,
$A_i, B_{ij}^k, K_{ij}$ are given $d\times d$ matrices with Lebesgue  measurable essentially bounded on the semi-axis $[0,\infty)$
entries, and the delays $h_i,h_{ij}^k, g_{ij}$ are Lebesgue measurable functions.

Let us denote by $\| \cdot \|$ an arbitrary vector norm,
the same notation will be used for the induced matrix norm,
$I$ is the identity matrix. The absolute value of the matrix is understood componentwise: 
if $B=\left(b_{ij}\right)_{i,j=1}^d$ then $|B|:=\left(|b_{ij}|\right)_{i,j=1}^d$,
we use the same notation for vectors. 
Matrix or vector inequalities $A\leq B$ or $A<B$ are also understood componentwise: each entry satisfies the inequality.
All the relations (equations and inequalities) with Lebesgue measurable functions are understood almost everywhere.

To get main results, we will utilize the matrix measure (the logarithmic norm)
defined in~\cite{S,Z} and~\cite[Table 3.1, p. 286]{KM} as
\begin{equation*}
\mu(B):=\lim_{\vartheta\rightarrow 0^+}\frac{\|I+\vartheta B\|-1}{\vartheta} \, .
\end{equation*}
In particular, for the maximum norm ${\|x\|}_{\infty}:=\max_{i=1,\dots, d}|x_i|$ we have for $B=(b_{ij})_{i,j=1}^d$
$${\|B\|}_{\infty}:=\max_{i=1,\dots,d}\sum_{j=1}^{d} |b_{ij}|, ~~
\mu_{\infty}(B)=\max_{i=1\dots,d}\left\{b_{ii}+\sum_{j=1,j\neq i}^{d} |a_{ij}|\right\}.
$$

The following classical definition for an $M$-matrix will be used.

\begin{definition} \cite{Berman}
A matrix $B=(b_{ij})_{i,j=1}^n$ is called a (non-singular) {\em $M$-matrix} if $b_{ij}\leq 0, i\neq j$, and
one of the following equivalent conditions holds:
\begin{enumerate}
\item
$B$ is invertible, and $B^{-1} \geq 0$;
\item
the leading principal minors of the matrix $B$ are positive.
\end{enumerate}
\end{definition}

\begin{remark}
\label{obvious}
Note that for any three column vectors $X\leq Y$ and $Z\geq 0$ we have $Z^T(Y-X) \geq 0$ as a sum of non-negative numbers. Thus, for any matrix $A\geq 0$ and $X\leq Y$ with entries of arbitrary signs, we have $AX \leq AY$.
\end{remark}

For a fixed bounded interval $\Omega=[t_0,t_1]$, let  $\|y\|_{\Omega}= \mbox{ess}\sup_{t\in \Omega} \|y(t)\|$, 
also for a half-line $\|f\|_{[t_0,\infty)}=\mbox{ess}\sup_{t\geq t_0} \|f(t)\|$.
Denote by ${{\bf L}}_{\infty}^d[t_0,t_1]$ the space of all essentially bounded on $\Omega$
vector functions $y$ with the norm $\|y\|_{\Omega}$, the same for ${{\bf L}}_{\infty}^d[t_0,\infty)$
and for the spaces ${{\bf L}}_{\infty}^{d \times d} [t_0,t_1]$, ${{\bf L}}_{\infty}^{d \times d} [t_0,\infty)$ of matrix functions.

Assume that for a fixed $t_0\geq 0$ vector-functions 
$
\varphi_i
\colon(-\infty,t_0]\rightarrow {\mathbb R}^d$, $i=1, \dots, n
$
are Borel measurable and bounded.
Below for every $t_0\geq 0$,  system \eqref{2.1} will be considered with the initial condition 
\begin{equation}\label{2.2}
x_i(t)=\varphi_i(t),\,\,\,\,t\leq t_0, ~~i=1, \dots, n 
\end{equation}
under the hypotheses: 
\begin{itemize}
\item[$(i)$] 
$A_i, B_{ij}^{k},K_{ij} \colon [0,\infty)\to\mathbb{R}^{d\times d}$, $k=1,\dots,m_{ij}$,
$i,j=1,\dots,n$  belong to ${{\bf L}}_{\infty}^{d \times d} [0,\infty)$.

\item[$(ii)$] 
$ h_{i}, h_{ij}^k, g_{ij} \colon [0,\infty)\to\mathbb{R}$,
are Lebesgue measurable functions,
and there are numbers $\tau_{ij}^k>0, \sigma_{ij}>0, \tau_{i}\geq 0$ such that almost everywhere
\begin{equation}\label{2.3}
0 \leq t-h_{ij}^k(t)\leq \tau_{ij}^k,~ 0 \leq t-g_{ij}(t)\leq \sigma_{ij}, ~0 \leq t-h_{i}(t)\leq \tau_{i},
~k=1,\dots,m_{ij}, ~i,j=1,\dots,n, ~t\geq 0.
\end{equation}
\end{itemize}
A solution of problem~\eqref{2.1},~\eqref{2.2} is understood
in the sense of the following definition.

\begin{definition}\label{definition2.1}
A set of vector-functions $x_i\colon {\mathbb R}\rightarrow {\mathbb R}^d, ~i=1,\dots,n$
is called {\bf a solution of problem}~\eqref{2.1},~\eqref{2.2} if $x_i$
satisfy~\eqref{2.1} for almost
all $t\in [t_0,\infty)$
and \eqref{2.2} for all $t\in [t_0,t_0-\tilde{\tau}_i]$, 
where $\displaystyle \tilde{\tau}_i = \max\left\{ \max_{j,k} \tau_{ij}^k, \max_{j} \sigma_{ij} , \tau_i \right\}$, $i=1, \dots, n$.
\end{definition}

Let $x_i,i=1,\dots,n$ be a solution of problem~\eqref{2.1},~\eqref{2.2}.
Introduce the matrix function $\Phi(t)=\{\varphi_1(t),\dots,\varphi_n(t)\}$ with $\varphi_i$ as columns.

\begin{definition}\label{definition2.2}
System~\eqref{2.1} is called {\bf uniformly exponentially stable}, if there exist
positive constants $H$ and $\nu$,  such that  any solution
$X$
of~\eqref{2.1},~\eqref{2.2} satisfies
$$
\|x_i(t)\|\leq H e^{-\nu (t-t_0)}\sup_{t\in [t_0,t_0-\max_i\tilde{\tau}_i]}\|\Phi(t)\|,\,\,\,t\geq t_ 0 \geq 0, ~~i=1, \dots, n,
$$
where $H$ and $\nu$ do not depend on $t_ 0$ and $\Phi$.
\end{definition}

Consider a non-homogeneous version of~\eqref{2.1}
\begin{equation}
\label{2.4}
\begin{array}{ll}
\dot{x_i}(t)= & \displaystyle A_i(t)x_i(h_i(t)) +\sum_{j=1}^n \sum_{k=1}^{m_{ij}} B_{ij}^k(t)x_j(h_{ij}^k(t))
\vspace{2mm} \\
& \displaystyle  + \sum_{j=1}^n\int\limits_{g_{ij}(t)}^t K_{ij}(t,s)x_j(s)ds+f_i(t),~i=1,\dots,n, ~t\geq t_0 \geq 0
\end{array}
\end{equation}
and assume that $f_i\in {{\bf L}}_{\infty}^d[t_0,\infty)$. 
The definition of a solution for  \eqref{2.4}, \eqref{2.2} is similar to that of \eqref{2.1}, \eqref{2.2}.

A modification of~\cite[Lemma 2]{BDSZ1}
will further be used.

\begin{uess}\label{lemma2.1}
Let a function $a\colon [t_0,\infty)\to [0,\infty)$
be Lebesgue measurable and $\omega\in {{\bf L}}_{\infty}^1[t_0,\infty)$. Then
the inequality
$\displaystyle
\left|\int_{t_0}^t e^{-\int_s^t a(\tau)d\tau}a(s)\omega(s)ds\right|
\leq {\mathrm{ess\,sup}}_{t\in[t_0,t_1]}|\omega(t)|,\,\,\,\,
t\in [t_0,t_1],
$
holds for any $t_1>t_0$. 
\end{uess}

Let $x_i$ be a solution of non-homogeneous system~\eqref{2.4}, where $f_i\in {\bf L}_{\infty}^d[t_0,\infty)$
satisfy
\begin{equation}\label{2.5}
x_i(t)=0,\,\,\,\,t\le t_0, ~~i=1, \dots, n.
\end{equation}

The following Bohl-Perron type result is cited from~\cite{AS, Gil}.

\begin{uess}\label{lemma2.2}
If, for any  $t_0\geq 0$ and any 
$f_i\in {{\bf L}}_{\infty}^d[t_0,\infty)$, $i=1, \dots, n$, all $x_i$ in the solution
of  problem~\eqref{2.4}, \eqref{2.5} belong to ${{\bf L}}_{\infty}^d[t_0,\infty)$, $i=1, \dots, n$
then system~\eqref{2.1} is uniformly exponentially stable.
\end{uess}

Below, we use the Coppel inequality~\cite{Desoer,Z}
for a system of ordinary differential equations 
\begin{equation}\label{2.6}
\dot{x}(t)=D(t)x(t),\,\,\,\,\,t\ge 0,
\end{equation}
where the columns of $D$ are in ${{\bf L}}_{\infty}^d[0,\infty)$.
Let $\Phi (t,s)$  be the fundamental matrix of~\eqref{2.6},
i.e. a solution of the problem 
$$
\dot{x}(t)=D(t)x(t),\,\,\,\,\,t>s\ge 0,\,\,\,\,\,x(s)=I.
$$

\begin{uess}[Coppel inequality]\label{lemma2.3}
The fundamental matrix $\Phi(t,s)$ of~\eqref{2.6} satisfies
\begin{equation}\label{2.7}
\left\| \Phi(t,s)\right\|\leq \exp \left({\int_s^t \mu(D (\xi)) d\xi}\right),\,\,\,\,\,t>s\ge 0.
\end{equation}
\end{uess}

\begin{remark}\label{remark2.1}
By Lemma~\ref{lemma2.3}, the condition $\mu(D(t))\leq d_0<0$, $t\in[t_0,\infty)$, $\forall t_0 \ge 0$ for the matrix measure implies
uniform exponential stability of system~\eqref{2.6}.
\end{remark}

The final auxiliary result gives an a priori estimate for the derivative of a solution.
Let $\Omega :=[t_0,t_1]\subset [t_0,\infty)$, $t_1>t_0$ and
\begin{equation*}
C_{ij}(t) :=\int_{g_{ij}(t)}^t |K_{ij}(t,s)|ds.
\end{equation*}

\begin{uess}\label{lemma2.4}
Let $x_i=x_i(t)$, $i=1,\dots,n$ be a solution of problem~\eqref{2.4},\eqref{2.5}, $\Omega=[t_0,t_1]$.
Then the derivatives $\dot{x_i}$ satisfy
\begin{equation}\label{2.10}
{\|\dot{x_i}\|}_{\Omega} \leq \|A_i\|_{[t_0,\infty)}{\|x_i\|}_{\Omega}+\sum_{j=1}^n \left(\sum_{k=1}^{m_{ij}} \|B_{ij}^k\|_{[t_0,\infty)}
+ \|C_{ij}\|_{[t_0,\infty)}\right){\|x_j\|}_{\Omega} +\|f_i\|_{[t_0,\infty)}.
\end{equation}
\end{uess}

\begin{proof}
Let $t\in \Omega$ and $x=x(t)$ be a solution of~\eqref{2.4}, \eqref{2.5}.
By \eqref{2.5},  $x(t)=0$ if $t\le t_0$.
Estimating the norm in the left-hand side of~\eqref{2.4}, we derive
\begin{equation}\label{2.11}
\|\dot{x_i}(t)\|\leq \|A_i\|_{[t_0,\infty)}{\|x_i\|}_{\Omega}+\sum_{j=1}^n \left(\sum_{k=1}^{m_{ij}} \|B_{ij}^k\|_{[t_0,\infty)}
+ \|C_{ij}\|_{[t_0,\infty)}\right){\|x_j\|}_{\Omega} + \| f_i \|_{\Omega}.
\end{equation}
Since the right-hand side of~\eqref{2.11} does not depend on $t\in \Omega$, 
inequality~\eqref{2.11}
implies 
\eqref{2.10}.
\end{proof}

\section{Main Results}\label{main_result}

\subsection{The main theorem}

\begin{guess}\label{theorem3.1}
Let there exist $\alpha_i$ such that $\mu(A_i(t))\leq \alpha_i<0$, $i=1,\dots,n$ and $L=I-D$
be an $M-$matrix, where the entries of $D=(d_{ij})_{i,j=1}^n$ be defined as follows:
$$
\begin{array}{ll}
d_{ii}=& \displaystyle \tau_i\left\|\frac{A_i}{\mu(A_i)}\right\|_{[t_0,\infty)}
\left(\|A_i\|_{[t_0,\infty)}+\sum_{k=1}^{m_{ii}} \|B_{ii}^k\|_{[t_0,\infty)}+\|C_{ii}\|_{[t_0,\infty)}\right)
\displaystyle +\sum_{k=1}^{m_{ii}}\left\|\frac{B_{ii}^k}{\mu(A_i)}\right\|_{[t_0,\infty)} \!\!\!\!\!\!
+\left\|\frac{C_{ii}}{\mu(A_i)}\right\|_{[t_0,\infty)},
\vspace{2mm} \\
d_{ij}= & \displaystyle \tau_i\left\|\frac{A_i}{\mu(A_i)}\right\|_{[t_0,\infty)}\left(\sum_{k=1}^{m_{ij}} \|B_{ij}^k\|_{[t_0,\infty)}+\|C_{ij}\|_{[t_0,\infty)}\right)
+\sum_{k=1}^{m_{ij}}\left\|\frac{B_{ij}^k}{\mu(A_i)}\right\|_{[t_0,\infty)}+\left\|\frac{C_{ij}}{\mu(A_i)}\right\|_{[t_0,\infty)},
~~j\neq i.
\end{array}
$$
Then system \eqref{2.1} is uniformly exponentially stable.
\end{guess}

\begin{proof}
To apply Lemma~\ref{lemma2.2}, we explore boundedness of solutions to \eqref{2.4},\eqref{2.5}.
First, transform \eqref{2.4} to the form
$$
\dot{x_i}(t)= 
 A_i(t)x_i(t) -A_i(t)\int\limits_{h_i(t)}^t \dot{x_i}(s)ds+\sum_{j=1}^n \sum_{k=1}^{m_{ij}} B_{ij}^k(t)x_j(h_{ij}^k(t))
+\sum_{j=1}^n \int\limits_{g_{ij}(t)}^t K_{ij}(t,s)x_j(s)ds+f_i(t).
$$
Denote by $\Phi_i(t,s)$ the fundamental matrix of 
\begin{equation}\label{3.1}
\dot{x}(t)=A_i(t)x(t).
\end{equation}
Since $\mu(A_i(t))\leq \alpha_i<0$, \eqref{3.1} is uniformly exponentially stable.
Continue transformations of equation \eqref{2.4}
$$
\begin{array}{ll}
x_i(t)= & \displaystyle \int_{t_0}^t \Phi_i(t,s)\left[-A_i(s)\int\limits_{h_i(s)}^s \dot{x_i}(\xi)d\xi
+\sum_{j=1}^n \sum_{k=1}^{m_{ij}} B_{ij}^k(s)x_j(h_{ij}^k(s))\right.
\\
 & \displaystyle \left.+   \sum_{j=1}^n  \int\limits_{g_{ij}(s)}^s K_{ij}(s,\xi)x_j(\xi)d\xi\right]ds+\tilde{f_i}(t),
\end{array}
$$
where $\tilde{f_i}(t)=\int_{t_0}^t \Phi_i(t,s) f_i(s)ds \in {{\bf L}}_{\infty}^d[t_0,\infty)$. 
Hence by Lemma~\ref{lemma2.3},
$$
\begin{array}{lll}
\|x_i(t)\| & \leq  & \displaystyle \int_{t_0}^t \left. \left. e^{\int_s^t \mu(A_i(\xi))d\xi} \right| \mu(A_i(s)) \right| \left[\left\|\frac{A_i(s)}{\mu(A_i(s))}\right\|
\int\limits_{h_i(s)}^s \|\dot{x_i}(\xi)\|d\xi
\right.
\vspace{2mm} \\
& & + \displaystyle  \sum_{j=1}^n \sum_{k=1}^{m_{ij}}\left\|\frac{B_{ij}^k(s)}{\mu(A_i(s))}\right\|\|x_j(h_{ij}^k(s))\|
\left. \displaystyle +\sum_{j=1}^n \left\|\frac{C_{ij}(s)}{\mu(A_i(s))}\right\|\max_{\xi\in [g_{ij}(s),s]}\|x_j(\xi)\|\right]ds+\|\tilde{f_i}(t)\|
\vspace{2mm} \\
& \leq & \displaystyle  \tau_i\left\|\frac{A_i}{\mu(A_i)}\right\|_{[t_0,\infty)}\|\dot{x_i}\|_{\Omega}+\|\tilde{f_i}\|_{[t_0,\infty)}
+\sum_{j=1}^n \left(\sum_{k=1}^{m_{ij}}\left\|\frac{B_{ij}^k}{\mu(A_i)}\right\|_{[t_0,\infty)}
+ \left\|\frac{C_{ij}}{\mu(A_i)}\right\|_{[t_0,\infty)}\right)\|x_j\|_{\Omega}.
\end{array}
$$
Therefore  by Lemma~\ref{lemma2.4},
$$
\begin{array}{lll}
\|x_i\|_{\Omega} & \leq  & \displaystyle 
\tau_i\left\|\frac{A_i}{\mu(A_i)}\right\|_{[t_0,\infty)}\left(\|A_i\|_{[t_0,\infty)}{\|x_i\|}_{\Omega}
+\sum_{j=1}^n \left(\sum_{k=1}^{m_{ij}} \|B_{ij}^k\|_{[t_0,\infty)}+\|C_{ij}\|_{[t_0,\infty)}\right){\|x_j\|}_{\Omega}\right)
\vspace{2mm} \\
& & \displaystyle +\sum_{j=1}^n \left(\sum_{k=1}^{m_{ij}}\left\|\frac{B_{ij}^k}{\mu(A_i)}\right\|_{[t_0,\infty)}
+ \left\|\frac{C_{ij}}{\mu(A_i)}\right\|_{[t_0,\infty)}\right)\|x_j\|_{\Omega}+M_i
\vspace{2mm} \\
& = & \displaystyle d_{ii}\|x_i\|_{\Omega} +\sum_{j\neq i} d_{ij}\|x_j\|_{\Omega} +M_i,
\end{array}
$$
where $M_i>0$ are constants, not dependent on $t_1>t_0$.

Introduce the column vectors $X_{\Omega}=(\|x_1\|_{\Omega},\dots,\|x_n\|_{\Omega})^T,M=(\|M_1\|,\dots,\|M_n\|)^T$. Then $(I-D)X_\Omega\leq M$.
Since $I-D$ is an $M$-matrix, there exists  $(I-D)^{-1}\geq 0$.
Therefore (see Remark~\ref{obvious}) $X_{\Omega}\leq (I-D)^{-1}M$, where the right-hand side of the last inequality does not depend on the interval~$\Omega$.
Hence the solution of problem \eqref{2.4},\eqref{2.5} is bounded on $[t_0,\infty)$.
By Lemma \ref{lemma2.2}, system \eqref{2.1} is uniformly exponentially stable.
\end{proof}

\subsection{Corollaries of the main theorem}

Consider several particular cases of \eqref{2.1}.

We start with the system where the first term in each vector equation is non-delayed
\begin{equation}\label{3.2}
\dot{x_i}(t)=A_i(t)x_i(t) +\sum_{j=1}^n \sum_{k=1}^{m_{ij}} B_{ij}^k(t)x_j(h_{ij}^k(t))
+\sum_{j=1}^n  \int\limits_{g_{ij}(t)}^t K_{ij}(t,s)x_j(s)ds,~i=1,\dots,n.
\end{equation}

\begin{corollary}\label{c3.1}
Assume that $\mu(A_i(t))\leq \alpha_i<0$, $i=1,\dots,n$ and $L=I-D$
is an $M-$matrix, where $D=(d_{ij})_{i,j=1}^n$ has the entries
$$
d_{ij}=\sum_{k=1}^{m_{ij}}\left\|\frac{B_{ij}^k}{\mu(A_i)}\right\|_{[t_0,\infty)}+\left\|\frac{C_{ij}}{\mu(A_i)}\right\|_{[t_0,\infty)},~i,j=1,\dots,n.
$$
Then system \eqref{3.2} is uniformly exponentially stable.
\end{corollary}

Next, consider a system without integral terms and with single delays in the diagonal terms
\begin{equation}\label{3.4}
\dot{x_i}(t)=A_i(t)x_i(h_i(t)) +\sum_{j\neq i} \sum_{k=1}^{m_{ij}} B_{ij}^k(t)x_j(h_{ij}^k(t)),~i=1,\dots,n.
\end{equation}

\begin{corollary}\label{c3.2}
Assume that $\mu(A_i(t))\leq \alpha_i<0, i=1,\dots,n$, and the matrix $I-D$
is an $M-$matrix, where $D=(d_{ij})_{i,j=1}^n$ are defined as
$$
d_{ii} 
= \displaystyle \tau_i\left\|\frac{A_i}{\mu(A_i)}\right\|_{[t_0,\infty)}\|A_i\|_{[t_0,\infty)},
~
d_{ij} 
= \tau_i\left\|\frac{A_i}{\mu(A_i)}\right\|_{[t_0,\infty)}\sum_{k=1}^{m_{ij}} \|B_{ij}^k\|_{[t_0,\infty)}
+\sum_{k=1}^{m_{ij}}\left\|\frac{B_{ij}^k}{\mu(A_i)}\right\|_{[t_0,\infty)},
~j\neq i.
$$
Then system \eqref{3.4} is uniformly exponentially stable.
\end{corollary}

We proceed to the case of two vector equations:
\begin{equation}\label{3.5}
\begin{array}{l}
\displaystyle
\dot{x_1}(t)=A_1(t)x_1(h_1(t)) +\sum_{j=1}^2 \sum_{k=1}^{m_{1j}} B_{1j}^k(t)x_j(h_{1j}^k(t))+
\sum_{j=1}^2 \int_{g_{1j}(t)}^t K_{1j} (t,s) x_j(s)ds, \vspace{2mm} \\
\displaystyle
\dot{x_2}(t)=A_2(t)x_2(h_2(t)) +\sum_{j=1}^2 \sum_{k=1}^{m_{2j}} B_{2j}^k(t)x_j(h_{2j}^k(t))+
\sum_{j=1}^2 \int_{g_{2j}(t)}^t K_{2j} (t,s) x_j(s)ds.
\end{array}
\end{equation}

\begin{corollary}\label{c3.4}
Let $\mu(A_i(t))\leq \alpha_i<0, i=1,2$ and 
\begin{equation}\label{3.6}
d_{11}<1, ~d_{22}<1, ~d_{12}d_{21}<(1-d_{11})(1-d_{22}),
\end{equation}
where
$$
\begin{array}{ll}
d_{11}= & \displaystyle \tau_1\left\|\frac{A_1}{\mu(A_1)}\right\|_{[t_0,\infty)}\left(\|A_1\|_{[t_0,\infty)}
+\sum_{k=1}^{m_{11}}\left\|B_{11}^k \right\|_{[t_0,\infty)} \!\!\!\!\!\!+\left\|C_{11} \right\|_{[t_0,\infty)}\right)
+\sum_{k=1}^{m_{11}}\left\|\frac{B_{11}^k}{\mu(A_1)}\right\|_{[t_0,\infty)}
\!\!\!\!\!\! +\left\|\frac{C_{11}}{\mu(A_1)}\right\|_{[t_0,\infty)},
\\
d_{22}= &  \displaystyle \tau_2\left\|\frac{A_2}{\mu(A_2)}\right\|_{[t_0,\infty)}\left(\|A_2\|_{[t_0,\infty)}
+\sum_{k=1}^{m_{22}}\left\|B_{22}^k \right\|_{[t_0,\infty)} \!\!\!\!\!\!+\left\|C_{22} \right\|_{[t_0,\infty)}\right)
+\sum_{k=1}^{m_{22}}\left\|\frac{B_{22}^k}{\mu(A_2)}\right\|_{[t_0,\infty)}
\!\!\!\!\!\! +\left\|\frac{C_{22}}{\mu(A_2)}\right\|_{[t_0,\infty)},
\end{array}
$$
$$
\begin{array}{ll}
d_{12}= &  \displaystyle \tau_1\left\|\frac{A_1}{\mu(A_1)}\right\|_{[t_0,\infty)}\left(\sum_{k=1}^{m_{12}}\left\|B_{12}^k \right\|_{[t_0,\infty)}
+\left\|C_{12} \right\|_{[t_0,\infty)}\right)
+ \displaystyle \sum_{k=1}^{m_{12}}\left\|\frac{B_{12}^k}{\mu(A_1)}\right\|_{[t_0,\infty)}
+\left\|\frac{C_{12}}{\mu(A_1)}\right\|_{[t_0,\infty)},
\\
d_{21}= &  \displaystyle \tau_2\left\|\frac{A_2}{\mu(A_2)}\right\|_{[t_0,\infty)}\left(\sum_{k=1}^{m_{21}}\left\|B_{21}^k \right\|_{[t_0,\infty)}
+\left\|C_{21} \right\|_{[t_0,\infty)}\right)
+\sum_{k=1}^{m_{21}}\left\|\frac{B_{21}^k}{\mu(A_2)}\right\|_{[t_0,\infty)}
+\left\|\frac{C_{21}}{\mu(A_1)}\right\|_{[t_0,\infty)}.
\end{array}
$$
Then system \eqref{3.5} is uniformly exponentially stable.
\end{corollary}

For only two delays in each vector equation in \eqref{3.5}, we get
\begin{equation}\label{3.7}
\begin{array}{lll}
\dot{x_1}(t) & = & A_1(t)x_1(h_1(t)) +B_{12}(t)x_2(h_{12}(t)), \\ 
\dot{x_2}(t) & = & A_2(t)x_2(h_2(t)) +B_{21}(t)x_1(h_{21}(t)).
\end{array}
\end{equation}

\begin{corollary}\label{c3.5}
Let $\mu(A_i(t))\leq \alpha_i<0$, $i=1,2$ and there exist constant matrices $\tilde{A_1},\tilde{A_2}$, $\tilde{B}_{12}$, $\tilde{B}_{21}$, such that
$|A_i(t)|\leq \tilde{A_i}$, $i=1,2$, $|B_{12}(t)|\leq \tilde{B}_{12}$, $|B_{21}(t)|\leq \tilde{B}_{21}$.
If 
\begin{equation}\label{3.8}
\begin{array}{l}
 \displaystyle  \tau_i\|\tilde{A_i}\|^2<|\alpha_i|, i=1,2, 
\vspace{2mm} \\ \displaystyle 
\|\tilde{B}_{12}\|\|\tilde{B}_{21}\|(1+\tau_1 \|\tilde{A_1}\|)(1+\tau_2 \|\tilde{A_2}\|)
<(|\alpha_1|-\tau_1\|\tilde{A_1}\|^2)(|\alpha_2|-\tau_2\|\tilde{A_2}\|^2)
\end{array}
\end{equation}
then system \eqref{3.7} is uniformly exponentially stable.
\end{corollary}
\begin{proof}
Let us check that conditions of Corollary \ref{c3.4} hold, where  $B_{ii} \equiv 0, C_{ij}\equiv 0$. 
For brevity, we omit the interval index $[t_0,\infty)$ for the norms of matrix functions.
We have 
$$
d_{ii}=\tau_i\left\|\frac{A_i}{\mu(A_i)}\right\|\|A_i\|,~i=1,2,~
$$$$
d_{12}= \tau_1\left\|\frac{A_1}{\mu(A_1)}\right\|\|B_{12}\|+\left\|\frac{B_{12}}{\mu(A_1)}\right\| ,~
d_{21}= \tau_2\left\|\frac{A_2}{\mu(A_2)}\right\|\|B_{21}\|+\left\|\frac{B_{21}}{\mu(A_2)}\right\|.
$$ 

The inequality $|A| \leq B$ implies $\| |A| \| \leq \| B \|$ for any matrix norm. Hence
\begin{equation}\label{3.9}
d_{ii}\leq \tau_i\frac{\|\tilde{A_i}\|^2}{|\alpha_i|}, i=1,2,
\end{equation}
\begin{equation}\label{3.10}
d_{12}\leq \frac{\|\tilde{B}_{12}\|(1+\tau_1 \|\tilde{A_1}\|)}{|\alpha_1|},~
d_{21}\leq \frac{\|\tilde{B}_{21}\|(1+\tau_2 \|\tilde{A_2}\|)}{|\alpha_2|}.
\end{equation}
Inequality \eqref{3.9} and the first inequality in \eqref{3.8} imply $d_{ii}<1$, $i=1,2$. By \eqref{3.10}, inequality $d_{12}d_{21}<(1-d_{11})(1-d_{22})$ holds if
$$
\frac{\|\tilde{B}_{12}\|\|\tilde{B}_{21}\|(1+\tau_1 \|\tilde{A_1}\|)(1+\tau_2 \|\tilde{A_2}\|)}{|\alpha_1||\alpha_2|}<
\left(1-\frac{\tau_1\|\tilde{A_1}\|^2}{|\alpha_1|}\right)\left(1-\frac{\tau_2\|\tilde{A_2}\|^2}{|\alpha_2|}\right),
$$
which is equivalent to the second inequality in \eqref{3.8}.
Since all the conditions in \eqref{3.6} hold, by Corollary \ref{c3.4} system \eqref{3.7} is uniformly exponentially stable.
\end{proof}

\begin{remark}\label{remark3.1}
Theorem~\ref{theorem3.1} and its corollaries also hold for vector differential equations 
if the terms $\displaystyle \int\limits_{g_{ij}^k(t)}^t K_{ij} (t,s)x_j(s)ds$
in equation \eqref{2.1} are replaced with $\displaystyle C_{ij} (t) \int\limits_{g_{ij} (t)}^t d_s R_{ij} (t,s)x_j(s)$, where distributed delays are of a more general type, and
$\displaystyle  \int\limits_{g_{ij}(t)}^t d_s |R_{ij}(t,s)|=I$.
\end{remark}

\section{Applications to Second Order Vector Delay Differential Equations}

\subsection{Linear equations}

In this section we consider a second order vector equation
\begin{equation}\label{4.1}
\ddot{x}(t)+A(t)\dot{x}(t)+B(t)x(h(t))=0,~~ t\geq t_0\geq 0,
\end{equation}
where for the matrices $A, B$ and the delay $h$ the same conditions  hold as $(i), (ii)$ for system \eqref{2.1}.
Definitions of solutions of an initial value problem for vector equation \eqref{4.1} 
and of uniform exponential stability are also similar to the definitions for \eqref{2.1},
and we omit them. 

\begin{guess}\label{theorem4.1}
Assume that for some $t\geq t_0\geq 0$, there exist a constant matrix $\tilde{A}$ and  numbers $\alpha<0$, $\tau\geq 0$ such that 
$t-h(t)\leq \tau$, $\mu(-\tilde{A})<0$, $\mu\left(\tilde{A}-2A(t)\right)\leq \alpha<0$
and 
\begin{equation}\label{4.0}
\frac{1}{|\mu(-\tilde{A})|}\left(\left\|\frac{2\tilde{A}A-\tilde{A}^2-4B}{\mu(\tilde{A}-2A)}\right\|_{[t_0,\infty)} \!\!\!\!\!\!
+2\tau \left\|\frac{\tilde{A}B}{\mu(\tilde{A}-2A)}\right\|_{[t_0,\infty)}\right)
<1-2\tau\left\|\frac{B}{\mu(\tilde{A}-2A)}\right\|_{[t_0,\infty)}  \!\!.
\end{equation}
Then equation \eqref{4.1} is uniformly exponentially stable.
\end{guess}

\begin{proof}
Equation \eqref{4.1}  can be transformed to the form
\begin{equation}\label{4.2}
\ddot{x}(t)+A(t)\dot{x}(t)+B(t)x(t)-B(t)\int_{h(t)}^t \dot{x}(s)ds=0.
\end{equation}
After the substitution 
$$\dot{x}(t)=-\frac{\tilde{A}}{2}x(t)+y(t), ~~\ddot{x}(t)=\frac{\tilde{A}^2}{4}x(t)-\frac{\tilde{A}}{2}y(t)+ \dot{y}(t),$$ 
equation \eqref{4.2} becomes
$$
\frac{\tilde{A}^2}{4}x(t)-\frac{\tilde{A}}{2}y(t)+ \dot{y}(t)+A(t)\left[-\frac{\tilde{A}}{2}x(t)+y(t)\right]
+B(t)x(t)-B(t)\int\limits_{h(t)}^t \left(-\frac{\tilde{A}}{2}x(s)+y(s)\right)ds=0.
$$
Hence equation \eqref{4.1} is equivalent to the system
\begin{equation}\label{4.3}
\begin{array}{ll}
\dot{x}(t) = & \displaystyle  -\frac{\tilde{A}}{2}x(t)+y(t) \vspace{2mm} \\
\dot{y}(t) = & \displaystyle \left(\frac{\tilde{A}}{2}-A(t)\right)y(t)+\left(\frac{\tilde{A}}{2}A(t)-\frac{\tilde{A}^2}{4}-B(t)\right)x(t) \vspace{2mm} \\
& \displaystyle -\frac{\tilde{A}}{2}B(t)\int_{h(t)}^t x(s)ds+B(t)\int_{h(t)}^t y(s)ds=0.
\end{array}
\end{equation}
For system \eqref{4.3}, let us apply Corollary \ref{c3.4}. System \eqref{4.3} has a form of \eqref{3.5}, where
$$
x_1=x, ~x_2=y, ~~A_1=-\frac{\tilde{A}}{2}, ~~h_1(t)=t, ~\tau_1=0, ~~m_{11}=0, ~m_{12}=1, $$ $$B_{11}=0, ~B_{12}=1,
~~h_{12}(t)=t, ~~C_{11}=C_{12}=0, 
$$
$$
A_2=\frac{\tilde{A}}{2}-A(t), ~h_2(t)=t, ~\tau_2=0, ~~m_{21}=1, ~m_{22}=0, ~~B_{21}=\frac{\tilde{A}}{2}A(t)-\frac{\tilde{A}^2}{4}-B(t),
$$$$
C_{21}=\frac{\tilde{A}}{2}|B(t)|(t-h(t))\leq \frac{\tilde{A}}{2}|B(t)|\tau,~~
C_{22}=|B(t)|(t-h(t))\leq |B(t)|\tau.
$$
We have 
$$
\mu(A_1)<0,~\mu(A_2)<0,~~ d_{11}=0, ~d_{22}\leq \tau\left\|\frac{B}{\mu(\frac{\tilde{A}}{2}-A)}\right\|_{[t_0,\infty)}, 
d_{12}=\frac{2}{|\mu(-\tilde{A})|},
$$$$
d_{21}\leq \left\|\frac{\frac{\tilde{A}}{2}A-\frac{\tilde{A}^2}{4}-B}{\mu(\frac{\tilde{A}}{2}-A)}\right\|_{[t_0,\infty)}
+\tau \left\|\frac{\frac{\tilde{A}}{2}B}{\mu(\frac{\tilde{A}}{2}-A)}\right\|_{[t_0,\infty)}.
$$
Then $d_{12}d_{21} \leq (1-d_{11})(1-d_{22})$ takes the form
$$
\frac{2}{|\mu(-\tilde{A})|}\left(\left\|\frac{\frac{\tilde{A}}{2}A-\frac{\tilde{A}^2}{4}-B}{\mu(\frac{\tilde{A}}{2}-A)}\right\|_{[t_0,\infty)}
+\tau \left\|\frac{\frac{\tilde{A}}{2}B}{\mu(\frac{\tilde{A}}{2}-A)}\right\|_{[t_0,\infty)}\right)
<1-\tau\left\|\frac{B}{\mu(\frac{\tilde{A}}{2}-A)}\right\|_{[t_0,\infty)},
$$
which is equivalent to \eqref{4.0}. 
Hence all the conditions of Corollary~\ref{c3.4} hold for \eqref{4.3}. Thus, system \eqref{4.3}
and therefore equation \eqref{4.2} are uniformly exponentially stable.
\end{proof}

\begin{remark}\label{remark4.1}
If $A_1\leq A(t)\leq A_2$ then anyone of $A_1,A_2$ can be taken as $\tilde{A}$, see examples below.
\end{remark}

\begin{corollary}\label{c4.0}
Let  there exist a constant matrix $\tilde{A}$,  $\alpha<0$ and $\tau\geq 0$ such that 
for $t\geq t_0\geq 0$, $t-h(t)\leq \tau$, $\mu(-\tilde{A})<0$, $\displaystyle \mu\left(\tilde{A}-2A(t)\right)\leq \alpha<0$ and
\begin{equation}\label{4.a}
\|2\tilde{A}A-\tilde{A}^2-4B\|_{[t_0,\infty)}
+2\tau\|\tilde{A}B\|_{[t_0,\infty)}
<|\mu(-\tilde{A})| \left( |\alpha|-2\tau \|B\|_{[t_0,\infty)} \right).
\end{equation}
Then equation \eqref{4.1} is uniformly exponentially stable.
\end{corollary}

Consider now a second order system with constant coefficients
\begin{equation}\label{4.4}
\ddot{x}(t)+A\dot{x}(t)+Bx(h(t))=0.
\end{equation}

Choosing $\tilde{A}=A$, $\alpha = \mu(-A)$ in Corollary~\ref{c4.0}, we get a result for \eqref{4.4}.

\begin{corollary}\label{c4.1}
Let $t-h(t) \leq \tau<\infty$ for $t\geq t_0\geq 0$, 
$\mu(-A)<0$ and 
\begin{equation}\label{4.5}
\|A^2-4B\|+2\tau \|AB\|<|\mu(-A)|(|\mu(-A)|-2\tau \|B\|).
\end{equation}
Then equation \eqref{4.4} is uniformly exponentially stable.
\end{corollary}

Next, taking $h(t)=t$ in Theorem \ref{theorem4.1}, we get a result for a vector ordinary differential equation of the second order.

\begin{corollary}\label{c4.2}
Let  there exist a constant matrix $\tilde{A}$, $\alpha<0$ and $\tau\geq 0$ such that 
for $t\geq t_0\geq 0$, $\mu(-\tilde{A})<0$, $\displaystyle \mu\left(\tilde{A}-2A(t)\right)\leq \alpha<0$ and
\begin{equation}\label{4.5a}
\|2\tilde{A}A-\tilde{A}^2-4B\|_{[t_0,\infty)}
<|\alpha||\mu(-\tilde{A})|.
\end{equation}
Then equation \eqref{1.1} is uniformly exponentially stable.
\end{corollary}

\begin{corollary}\label{c4.3}
Assume that $\mu(-A)<0$ and $\|A^2-4B\|<(\mu(-A))^2$.
Then equation \eqref{1.1} with constant matrices $A$ and $B$ is uniformly exponentially stable.
\end{corollary}

\begin{example}\label{example4.1}
Consider equation \eqref{4.1}, where $t-h(t)\leq \tau$,
$\displaystyle
A(t)=\left(\begin{array}{cc}
4&\sin^2 t\\
\cos^2 t&6
\end{array}\right)$,
$\displaystyle
B(t)=\left(\begin{array}{cc}
4&2\sin^2 t\\
2\cos^2 t&8
\end{array}\right)$.
Choosing  $\displaystyle \tilde{A}:=\left(\begin{array}{cc}
4&1\\
1 &6
\end{array}\right)$, we have $\mu(-\tilde{A})=-3<0$.

$\tilde{A}-2A(t)=\left(\begin{array}{cc}
-4&1-2\sin^2 t\\
1-2\cos^2 t &-6
\end{array}\right),~~\mu(\tilde{A}-2A(t))\leq \alpha=-3<0$.
Next,
$$
2\tilde{A}A(t)= \left(\begin{array}{cc}
32+2\cos^2 t&12+8\sin^2 t\\
8+12\cos^2 t &72+2\sin^2 t
\end{array}\right),~~
\tilde{A}^2=\left(\begin{array}{cc}
17&10\\
10&37
\end{array}\right),~~
4B(t)=\left(\begin{array}{cc}
16&8\sin^2 t\\
8\cos^2 t &32
\end{array}\right),
$$$$
2\tilde{A}A(t)-\tilde{A}^2-4B(t)=
\left(\begin{array}{ll}
-1+2\cos^2 t&2\\
-2+4\cos^2 t &3+2\sin^2 t
\end{array}\right),
~~\tilde{A}B(t) =
\left(\begin{array}{ll}
16+2\cos^2 t&8+8\sin^2 t\\
4+12\cos^2 t &48+2\sin^2 t
\end{array}\right),
$$$$
\|2\tilde{A}A(t)-\tilde{A}^2-4B(t)\|_{[0,\infty)}=7,~\|\tilde{A}B\|_{[0,\infty)}=66, \|B\|_{[0,\infty)}=10.
$$
For $\alpha =-3$, inequality \eqref{4.a} is $7+132 \tau <3(3-20 \tau)$, which
holds if $\tau<\frac{1}{96}$. Thus for  $\tau<\frac{1}{96}$ vector equation \eqref{4.1} is 
uniformly exponentially stable.
\end{example}

Next, consider equation \eqref{4.1} with a non-zero right-hand side 
\begin{equation}\label{4.1a}
\ddot{x}(t)+A(t)\dot{x}(t)+B(t)x(h(t))=L(t),~~ t\geq t_0\geq 0.
\end{equation}

The following statement  can be justified for a variety of uniformly exponentially stable differential equations with bounded delays, 
we prove it here for second order vector equations.

\begin{uess}\label{lemma4.1}
Let equation \eqref{4.1} be uniformly exponentially stable and $L(t)$ have one of the following properties:

a) $L(t)$ is bounded, $t \in [t_0,\infty)$;

b) $\int_t^{t+1} \|L(s)\|ds$ is bounded, $t \in [t_0,\infty)$;

c) $\lim_{t\rightarrow\infty} L(t)=0$;

d)  $\lim_{t\rightarrow\infty} \int_t^{t+1} \|L(s)\|ds =0$.

Then any solution of equation \eqref{4.1a} possesses the same property.
\end{uess}
\begin{proof}
According to \cite[page 3]{BDK}, \cite{AMR}, there are matrix functions $X_1(t), X_2(t), X(t,s)$ such that
a solution $x$ of equation \eqref{4.1a} with the initial condition 
$x(t)=\varphi(t)$, $t \leq t_0$ satisfies 
$$
x(t)=X_1(t)x(t_0)+X_2(t)\dot{x}(t_0)-\int_{t_0}^{t_0+\tau} X(t,s)B(s)\varphi  (h(s))ds+\int_{t_0}^t X(t,s) L(s)ds,
$$
where $\varphi  (h(s))=0, ~h(s)\geq t_0$.
Exponential stability of \eqref{4.1} implies for some $M_i>0$, $\lambda_i>0$, $i=1,2$, $M>0$, $\lambda>0$,
$$
\|X_i(t)\|\leq M_i e^{-\lambda_i t}, ~~i=1,2,~~\|X(t,s)\|\leq M e^{-\lambda (t-s)}.
$$
Assume that condition b) holds: $\int_t^{t+1} \|L(s)\|ds\leq K$ for $t\geq t_0$.
We have
$$
\begin{array}{lll}
\|x(t)\| & \leq & \displaystyle M_1 e^{-\lambda_1 t} \|x(t_0)\|+ M_2 e^{-\lambda_2 t} \|\dot{x}(t_0)\|
+M \|B\|_{[t_0,\infty)}\|\varphi \|_{[t_0-\tau,t_0]}\int\limits_{t_0}^{t_0+\tau} e^{-\lambda (t-s)}ds 
\\ & & \displaystyle  +M\int\limits_{t_0}^t e^{-\lambda (t-s)} \|L(s)\|ds
\\
 & \leq & \displaystyle  M_1\|x(t_0)\|+M_2\|\dot{x}(t_0)\| +\frac{M}{\lambda}\|B\|_{[t_0,\infty)}\|\varphi \|_{[t_0-\tau,t_0]} e^{\lambda (t_0+\tau)}
\\
 &  &
\displaystyle  +M\int\limits_{t_0}^t e^{-\lambda (t-s)} \|L(s)\|ds.
\end{array}
$$

Let $t\in [t_0+m, t_0+m+1]$ for some integer $m \geq 0$.
Then
$$
\begin{array}{ll}
\displaystyle \int\limits_{t_0}^t e^{-\lambda (t-s)} \|L(s)\|ds & \leq \displaystyle \sum_{k=0}^m \int\limits_{t_0+k}^{t_0+k+1} e^{-\lambda (t-s)} \|L(s)\|ds
\leq \sum_{k=0}^m  e^{-\lambda (m-k)} \int\limits_{t_0+k}^{t_0+k+1} \|L(s)\|ds \\ & \displaystyle \leq K\sum_{k=0}^m e^{-\lambda (m-k)}
< \frac{K}{1-e^{-\lambda}},
\end{array}
$$
where the last expression does not depend on $m$. Hence the solution $x$ of equation \eqref{4.1a} is bounded, and $\int_t^{t+1} \|x(s)\|\, ds$ is bounded.
Therefore a) and b) yield that solutions of equation \eqref{4.1a} are bounded. Similarly, c)-d) imply convergence of all solutions to zero. 
\end{proof}

\begin{guess}\label{theorem4.3}
Let all the conditions of Theorem \ref{theorem4.1} hold and $L(t)$ have one of the  properties a)-d)
of Lemma \ref{lemma4.1}. Then any solution of equation \eqref{4.1a} possesses the same property.
\end{guess}

\subsection{Nonlinear equations}

In this section consider a nonlinear equation
\begin{equation}\label{4.7}
\ddot{x}(t)+A(t,x(t),\dot{x}(t))\dot{x}(t)+B(t,x(t),\dot{x}(t))x(h(t))=L(t),~~ t\geq t_0\geq 0,
\end{equation}
where  the  matrices $A(\cdot,u,v)$ and $B(\cdot,u,v)$
belong to ${{\bf L}}_{\infty}^{d\times d } [0,\infty)$ for any pair $(u,v) $ of locally integrable on $[0,\infty)$ vector functions,
and $L\in {{\bf L}}_{\infty}^d [0,\infty)$.
A solution of an initial value problem for \eqref{4.7} is defined similarly to that of \eqref{4.1}.
We assume that a solution exists and is unique.

Consider first equation \eqref{4.7} with $L(t)\equiv 0$.
\begin{guess}\label{theorem4.2}
Let for any $t \geq t_0\geq 0$, $L(t)\equiv 0$,  there exist a constant matrix $\tilde{A}$ and $\alpha<0$, $\tau\geq 0$ such that  
$t-h(t)\leq \tau$, $\mu(-\tilde{A})<0$, and for any pair $(u,v)$ of constant vectors, 
$\mu(\tilde{A}-2A(t,u,v))\leq \alpha<0$
and 
\begin{equation}
\label{4.8}
\begin{array}{ll}
& \displaystyle \frac{1}{|\mu(-\tilde{A})|}\left(\left\|\frac{2\tilde{A}A(\cdot,u,v)-\tilde{A}^2-4B(\cdot,u,v)}{\mu(\tilde{A}-2A(\cdot,u,v))}\right\|_{[t_0,\infty)}\right. \vspace{1mm} \\
+ & \left. 2\tau \left\|\frac{\tilde{A}B(\cdot,u,v)}{\mu(\tilde{A}-2A(\cdot,u,v))}\right\|_{[t_0,\infty)}\right)
<1-2\tau\left\|\frac{B(\cdot,u,v)}{\mu(\tilde{A}-2A(\cdot,u,v))}\right\|_{[t_0,\infty)}.
\end{array}
\end{equation}
Then  any solution $x$ of \eqref{4.7}  with $L \equiv 0$ tends to zero as $t\rightarrow\infty$, together with its derivative.
\end{guess}

\begin{proof}
Let $x$ be a solution of \eqref{4.7} with $L \equiv 0$. Consider the linear equation with variable coefficients for this fixed $x$:
\begin{equation}\label{4.9}
\ddot{y}(t)+A(t,x(t),\dot{x}(t))\dot{y}(t)+B(t,x(t),\dot{x}(t))y(h(t))=0, ~~t\geq t_0\geq 0.
\end{equation}
All conditions of Theorem~\ref{theorem4.1} hold for \eqref{4.9}. Hence this equation is uniformly exponentially stable.
Then any solution $y$ of this equation tends to zero, together with its derivative . But the vector-function $x$ is one of these solutions.
Therefore any solution $x$ of \eqref{4.7} tends to zero as $t\rightarrow\infty$, together with its derivative.
\end{proof}

Consider now equation \eqref{4.7} with a non-zero right-hand side $L(t)$.
\begin{corollary}\label{c4.4}
Let all the conditions of Theorem \ref{theorem4.2} hold and $L(t)$ possess one of Properties a)-d)
of Lemma~\ref{lemma4.1}.
Then any solution of equation \eqref{4.7} has the same property.
\end{corollary}

\begin{proof}
For any solution $x$ of \eqref{4.7}, consider the linear equation 
\begin{equation}\label{4.10}
\ddot{y}(t)+A(t,x(t),\dot{x}(t))\dot{y}(t)+B(t,x(t),\dot{x}(t))y(h(t))=L(t),~~ t\geq t_0\geq 0.
\end{equation}
Assume that $L(t)$ satisfies one of a)-d). Since equation \eqref{4.10}  with $L \equiv 0$ is exponentially stable, by Lemma~\ref{lemma4.1}
all solutions of \eqref{4.10}, as well as the solution $x$ of \eqref{4.7}, possess this property.
\end{proof}

\begin{corollary}\label{c4.5}
Let there exist a constant matrix $\tilde{A}$, $\alpha<0$ and $\tau\geq 0$ such that $
\mu(-\tilde{A})<0$, $t-h(t)\leq \tau$ for $t\geq t_0\geq 0$, 
and for any pair $(u,v)$ of constant vectors, 
 $\mu\left(\tilde{A}-2A(t,u,v)\right)\leq \alpha<0$ and 
\begin{equation}\label{4.11}
\begin{array}{ll}
& \displaystyle \|2\tilde{A}A(\cdot,u,v)-\tilde{A}^2-4B(\cdot,u,v)\|_{[t_0,\infty)}
+2\tau\|\tilde{A}B(\cdot,u,v)\|_{[t_0,\infty)}
\\ < & \displaystyle |\mu(-\tilde{A})|(|\alpha|-2\tau \|B(\cdot,u,v)\|_{[t_0,\infty)}).
\end{array}
\end{equation}
Then  any solution $x$ of \eqref{4.7}  with $L \equiv 0$ tend to zero as $t\rightarrow\infty$, together with its derivative.
\end{corollary}

\begin{example}\label{example4.2}
Consider equation \eqref{4.7}, where $L(t)\equiv 0$, $t-h(t)\leq \tau$, $x=(x_1,x_2)$, 
$$
A(t)=\left(\begin{array}{cc}
4&\sin^2 (tx_1(t))\\
\cos^2 (tx_2(t))&6
\end{array}\right),~~
B(t)=\left(\begin{array}{cc}
4&2\sin^2 (tx_1(t))\\
2\cos^2 (tx_2(t))&8
\end{array}\right).
$$
Exactly the same calculations as in Example \ref{example4.1} and condition \eqref{4.11} imply that for $\tau<\frac{1}{96}$,
any solution of \eqref{4.7} tends to zero as $t\rightarrow\infty$, together with its derivative.
\end{example}

\section{Discussion and Topics for Further Research}

We investigated uniform exponential stability for systems of linear vector functional differential equations,
as well as asymptotic stability of linear and certain nonlinear vector equations of the second order. 
System \eqref{2.1} has many applications to real-word models but, to the best of our knowledge, very little is known about its exponential stability.
One of most common applications is to second order vector delay differential equations, which is considered in the paper for
equations with one delay. However, since the general theory is developed for a more general model, the results are easily extended to equations with several bounded delays
\begin{equation}\label{5.1}
\ddot{x}(t)+A(t)\dot{x}(t)+\sum_{k=1}^m B_k(t)x(h_k(t))=0,~~ t\geq t_0\geq 0,
\end{equation}
and with a bounded distributed delay
\begin{equation}\label{5.2}
\ddot{x}(t)+A(t)\dot{x}(t)+\int_{h(t)}^t B(t,s)x(s)ds=0,~~ t\geq t_0\geq 0.
\end{equation}

Corollary \ref{c4.2} gives a new simple and general stability test for linear ordinary differential vector equations of the second order, 
compared to the papers \cite{Harris, Gil_a, Rugh, Sun, Zevin} where only matrices of a specific form were considered. 
For example, in the recent paper \cite{Gil_a} the following stability test was obtained.
Denote $\displaystyle A_R=\frac{A+A^*}{2},~A_I=\frac{A-A^*}{2i}$\,. 
\begin{proposition}\label{p5.0}
Let there exist $m_A>0$ such that
$$
A_R(t)\geq 2m_A I, ~T_R(t)\geq m_A I, ~\|T_R\|_{[0,\infty)}+\|T_I\|_{[0,\infty)}\leq 2 m_A^2,
$$
where $T(t)=m_A A(t)-B(t)$ and inequality $A\geq B$ means that  $A-B$ is a positive definite matrix (all eigenvalues are non-negative).
Then equation \eqref{1.1} is exponentially stable.
\end{proposition}
Actually it is not easy to check that all eigenvalues of time-variable matrix $T(t)$ are non-negative.

For a nonlinear delay vector equation of the second order 
\begin{equation}\label{5.4}
\ddot{X}(t)+F(X(t),\dot{X}(t))\dot{X}(t)+H(X(t-\tau))=0,
\end{equation}
where the matrix-function $F(u,v)$ and vector-function $H(u)$ are continuous,  an interesting explicit stability test was obtained in \cite{Tunc1}.

\begin{proposition}\cite{Tunc1}\label{p5.1}
Let there exist positive numbers $a_0,a_1,a_2$  such that
\\
1) the matrix $F(u,v)$ is symmetric, and the eigenvalues of this matrix $\lambda_i(F(u,v)) \geq a_1$ for all pairs $(u,v)\in {\mathbb R}^d\times {\mathbb R}^d$ .
\\
2) $H(0)=0$, $H(X)\neq 0$, $X\neq 0$, the Jacobian $J_H(X)$ is symmetric and $a_2  \leq \lambda_i(J_H(X))\leq a_0$.

If $\tau<\frac{a_1}{a_0\sqrt{d}}$ then the trivial solution of \eqref{5.4} is asymptotically stable.
\end{proposition}

Compare now Proposition~\ref{p5.1} of \cite{Tunc1} and the test in Theorem \ref{theorem4.2}.

The stability result of Proposition \ref{p5.1} applies only to autonomous equations,
with continuous functions $G,H$ and constant delay $\tau$. Our result works for non-autonomous equations, 
with measurable coefficients and a variable delay. We also apply different stability methods: in \cite{Tunc1}
the method of Krasovskii-Lyapunov functionals was applied, in this paper the stability results are based on Bohl-Perron theorem.

The linear equation
\begin{equation}\label{5.5}
\ddot{X}(t)+A\dot{X}(t)+Bx(t-\tau)=0,
\end{equation}
where $A,B$ are constant matrices, is a partial case of equation \eqref{5.4}. Let us compare for \eqref{5.5} the results of
Corollary~\ref{c4.1} 
and Proposition~\ref{p5.1}.
For the scalar case ($d=1$) Proposition~\ref{p5.1} leads to  $B>0$, $A>0$, $B\tau<A$,
while Corollary \ref{c4.1} reduces to $B>0, A>0, |A^2-4B|<A^2-4AB\tau$. In the scalar
case (and presumably also for small $d$), Proposition~\ref{p5.1} is better than Corollary~\ref{c4.1}.
But for large $d$, Corollary~\ref{c4.1} could be better than Proposition \ref{p5.1}.
Assume, for example, that in \eqref{5.5}  $A=aI, B=bI, a>0, b>0$. Then the condition of Proposition \ref{p5.1} becomes  $ b\tau<\frac{a}{\sqrt{d}}$,
where $\tau \rightarrow 0$ for fixed $a,b$ and $ d \rightarrow\infty$.
Proposition \ref{p5.1} gives the same condition as for the scalar case: $|a^2-4b|<a^2-4ab\tau$ and does not depend on $d$.
Consider for comparison $A=aI$, $B= \frac{1}{4}a^2 I$, where $a>0$. Then 
Proposition \ref{p5.1} implies asymptotic stability for $\tau<\frac{4}{a\sqrt{d}}$ compared to $\tau<\frac{1}{a}$ in Corollary~\ref{c4.1}. Thus Corollary~\ref{c4.1} gives a sharper result for $d>16$.

Together with stability results obtained here,  the aim of the paper was also to attract attention to the new class of systems of functional differential equations.

It is not yet clear how to obtain exponential stability tests for equations
with unbounded delays, for example, pantograph-type with $h(t)=\lambda t$, and for equations with delays in the derivative terms
such as
\begin{equation}\label{5.3}
\ddot{x}(t)+A(t)\dot{x}(g(t))+B(t)x(h(t))=0.
\end{equation}

Further, let us discuss other possible problems for future research.

An interesting extension of \eqref{2.1} is a nonlinear system
$$
\dot{x_i}(t)=A_i(t)x_i(h_i(t)) +\sum_{j=1}^n \sum_{k=1}^{m_{ij}} H_{ij}(t,x_j(h_{ij}^k(t)))
+ \sum_{j=1}^n F_{ij} \left( t,\int\limits_{g_{ij}(t)}^t K_{ij}(t,s)x_j(s)ds\right),~i=1,\dots,n.
$$
Results for this system can be applied to a nonlinear vector delay equation of the second order
$$
\ddot{x}(t)+F(t,\dot{x}(t))+H(t,x(h(t))=0.
$$
Also, there are no results on asymptotic behavior of solutions to neutral systems of vector differential equations
$$
\dot{x_i}(t)-Q_i(t)\dot{x}_i(g_i(t))=A_i(t)x_i(h_i(t)) +\sum_{j=1}^n \sum_{k=1}^{m_{ij}} B_{ij}^k(t)x_j(h_{ij}^k(t))
+ \sum_{j=1}^n\int\limits_{g_{ij}(t)}^t K_{ij}(t,s)x_j(s)ds,
$$
$i=1,\dots,n$ and neutral vector equations of the second order
$$
\ddot{x}(t)+A\dot{x}(g(t))+Bx(h(t))=C\ddot{x}(p(t)).
$$
It would be interesting to study stability for systems of second order equations even in the scalar case, such as
$$
\ddot{x}(t)+A_1\dot{x}(g_1(t))+B_1y(h_1(t))=0,
$$$$
\ddot{y}(t)+A_2\dot{y}(g_2(t))+B_2x(h_2(t))=0
$$
and
$$
\ddot{x}(t)+A_1\dot{x}(g_1(t))+B_1x(h_1(t))=C_1y(p_1(t)),
$$$$
\ddot{y}(t)+A_2\dot{y}(g_2(t))+B_2y(h_2(t))=C_2x(p_2(t)).
$$

\end{document}